\documentclass[12pt]{amsart}

\usepackage[utf8]{inputenc}
\usepackage{amsfonts}
\usepackage{amsmath}
\usepackage{graphicx}
\usepackage{float}
\usepackage[dvipsnames]{xcolor}
\usepackage{hyperref}
\usepackage{tikz-cd}

\usepackage{amssymb}
\usepackage{enumitem}
\usepackage{mathrsfs}

\usepackage{tikz}
\usetikzlibrary{topaths}
\usetikzlibrary{calc}

\usepackage{a4wide}

\usepackage{empheq}

\newtheorem{theorem}{Theorem}[section]
\newtheorem*{theorem*}{Theorem}
\newtheorem{lemma}[theorem]{Lemma}
\newtheorem{proposition}[theorem]{Proposition}
\newtheorem*{proposition*}{Proposition}
\newtheorem{corollary}[theorem]{Corollary}
\newtheorem*{corollary*}{Corollary}

\newtheorem{cit}[theorem]{Citation}
\newtheorem*{conjecture*}{Conjecture}
\newtheorem{question}[theorem]{Question}
\newtheorem*{question*}{Question}

\newtheorem{lettertheorem}{Theorem}

\theoremstyle{definition}
\newtheorem{definition}[theorem]{Definition}
\newtheorem*{definition*}{Definition}

\newtheorem{observation}[theorem]{Observation}

\newcommand{\N}{\mathbb{N}}
\newcommand{\Z}{\mathbb{Z}}

\newcommand{\R}{\mathbb{R}}

\DeclareMathOperator{\dlk}{\mathrm{lk}_\downarrow}
\DeclareMathOperator{\dflk}{\mathrm{lk}_\downarrow^\partial}
\DeclareMathOperator{\dclk}{\mathrm{lk}_\downarrow^\delta}
\DeclareMathOperator{\med}{\mathtt{med}}
\DeclareMathOperator{\diam}{\mathtt{diam}}
\DeclareMathOperator{\len}{\mathtt{len}}

\DeclareMathOperator{\F}{\mathrm{F}}
\DeclareMathOperator{\FP}{\mathrm{FP}}

\DeclareMathOperator{\Cay}{Cay}

\newcommand{\Rips}{\mathcal{VR}}

\numberwithin{equation}{section}

\begin{document}

\title{Word length, Morse theory, and Vietoris--Rips complexes}
\date{\today}


\author[S.~Hulbert]{Seth Hulbert}
\address{Department of Mathematics and Statistics, University at Albany (SUNY), Albany, NY}
\email{swhulbert@albany.edu}

\author[M.~C.~B.~Zaremsky]{Matthew C.~B.~Zaremsky}
\address{Department of Mathematics and Statistics, University at Albany (SUNY), Albany, NY}
\email{mzaremsky@albany.edu}

\begin{abstract}
We present a discrete Morse theoretic approach to proving high connectivity or contractibility of a Vietoris--Rips complex using distance to a fixed point as an initial measurement. In particular we focus on the case of a finitely generated group using word length. As a proof-of-concept application we prove that $\Rips_2(A_\Gamma)$ is contractible for $A_\Gamma$ a right-angled Artin group on a triangle-free graph $\Gamma$. We also prove an interesting sufficient condition for a group to be finitely presented that only requires checking connectivity of certain finite complexes.
\end{abstract}

\maketitle
\thispagestyle{empty}

\section{Introduction}\label{sec:intro}

The Vietoris--Rips complex $\Rips_t(X)$ of a metric space $X$ is the simplicial complex with vertex set $X$ and simplices determined by points being within distance $t$ of each other. This is a highly useful construction, in particular in geometric group theory, where a seminal result of Rips shows that for a hyperbolic group $G$, $\Rips_t(G)$ is contractible for large enough $t$. Since $G$ acts geometrically (i.e., properly and cocompactly) on $\Rips_t(G)$, this has implications for the topological properties of the group.

A question that has attracted much interest is, for which other groups does this happen, beyond hyperbolic ones? That is, given a finitely generated group $G$ with word metric coming from a finite generating set, is $\Rips_t(G)$ contractible for some $t<\infty$? Of course a necessary condition is that $G$ act geometrically on a contractible complex, but little else is known. Even the seemingly easy example of $\Z^n$ ($n\ge 2$) with the standard generating set was only recently handled, by Virk in \cite{virk25}, and in general the problem is wide open.

\begin{question}
Let $G$ be a finitely generated group that admits a proper cocompact action on some contractible simplicial complex. Must some $\Rips_t(G)$ ($t<\infty$) be contractible?
\end{question}

Note that this question implicitly asks for both a finite generating set and a value $t<\infty$ such that $\Rips_t(G)$ is contractible viewing $G$ with the word metric coming from that finite generating set. One can thus ask the stronger question of whether for every finite generating set such a $t$ exists.

\medskip

In this paper we set up a general sufficient criterion for $\Rips_t(G)$ to be highly connected or contractible, coming from Bestvina--Brady discrete Morse theory (see, e.g., \cite{bestvina97,zaremsky22}). This criterion is ``local'' in that it is about finite subsets of $G$ of a fixed diameter and at constant distance to the identity. Before stating the criterion, let us establish notation. Let $G$ be a finitely generated group with a fixed word metric $d$ coming from a finite generating set. For non-empty finite $S\subseteq G$ the \emph{diameter} $\diam(S)$ of $S$ is the maximum $d(g,g')$ for $g,g'\in S$. Let $\len$ be the word length function $\len(g)=d(g,1)$, and write $\len_{min}(S)$ for the minimum $\len$ value on $S$. Write $D_t(S)$ for the geometric realization of the poset of all finite $S'\supsetneq S$ (with the inclusion relation) such that $\diam(S')\le t$ and $\len_{min}(S')<\len_{min}(S)$. Now we can state the criterion, which is proved as Corollary~\ref{cor:general_morse}.

\begin{lettertheorem}\label{thrm:main_general}
Suppose $D_t(S)$ is $(n-|S|-1)$-connected for all non-empty finite $\{1\}\ne S\subseteq G$ with $\diam(S)\le t$ such that $\len$ is constant on $S$. Then $\Rips_t(G)$ is $(n-1)$-connected. If all such $D_t(S)$ are contractible then $\Rips_t(G)$ is contractible.
\end{lettertheorem}

This should be compared to the Morse theoretic approach in \cite{zaremsky22}, where roughly speaking (and ignoring many details) the poset one needs to understand instead consists of all $S'\supsetneq S$ with $\diam(S')=\diam(S)$. Our approach is in theory easier, since insisting not only on controlling the diameter but also on achieving $\len_{min}(S')<\len_{min}(S)$ makes for a smaller and hopefully easier to understand poset. We should also mention another flavor of discrete Morse theory that is often applied to Vietoris--Rips complexes, namely Forman discrete Morse theory, as in \cite{forman98}. As explained in \cite[Example~1.5 and Remark~1.6]{zaremsky22}, this is actually a special case of Bestvina--Brady discrete Morse theory, and some (but not all) of our work here could be phrased using the special case of Forman discrete Morse theory.

\medskip

As a proof-of-concept application, we consider the family of two-dimensional right-angled Artin groups with the standard word metric, and prove that they admit contractible Vietoris--Rips complexes, specifically for $t=2$.

\begin{lettertheorem}\label{thrm:main_raags}
If $\Gamma$ is triangle-free then $\Rips_2(A_\Gamma)$ (using the standard word metric) is contractible.
\end{lettertheorem}

As a remark, all right-angled Artin groups admit finite generating sets with respect to which they have contractible Vietoris--Rips complexes \cite[Theorem~4.1 and Lemma~5.20]{chalopin25}, so the focus here is on the standard generating set, where the contractibility question is open in general. We also remark that, using a different approach, Li and S\'anchez Salda\~na have independently proved the stronger result that if $\Gamma$ is triangle-free then $\Rips_t(A_\Gamma)$ (using the standard word metric) is contractible for all $t\ge 2$ \cite{li}. A natural prediction in general is that for $n$ the maximum size of a clique in $\Gamma$, equivalently the cohomological dimension of $A_\Gamma$, $\Rips_t(A_\Gamma)$ should be contractible for all $t\ge n$.

\medskip

We also use Theorem~\ref{thrm:main_general} to prove an interesting sufficient condition for a group to be finitely presented, which only requires one to prove that certain complexes are connected or non-empty, with no need to prove anything is simply connected.

\begin{lettertheorem}\label{thrm:fp}
Let $G$ be a finitely generated group with $d$ and $\len$ as above.
\begin{enumerate}
    \item Suppose there exists $t$ such that for all $1\ne g\in G$ the complex $D_t(\{g\})$ is connected, and for all $g,g'\in G$ with $d(g,g')=t$ and $\len(g)=\len(g')$ the complex $D_t(\{g,g'\})$ is non-empty. Then $G$ is finitely presented.
    \item If the Cayley graph of $G$ has no odd-length cycles, then $G$ is finitely presented as soon as there exists some odd $t\in\N$ such that $D_t(\{g\})$ is connected for all $g\ne 1$.
\end{enumerate}
\end{lettertheorem}

\subsection*{Acknowledgments} Aspects of this work will also appear in part of the first author's PhD thesis. We thank Kevin Li and Luis Jorge S\'anchez Salda\~na for helpful discussions and sharing their paper \cite{li} in advance, and the referee for several helpful comments.

\section{Vietoris--Rips complexes}\label{sec:VR}

Let $X$ be a metric space. In the traditional definition of the Vietoris--Rips complex $\Rips_t(X)$, the vertex set is $X$ and simplices are determined by elements of $X$ being within $t$ of each other. Here we will always use the barycentric subdivision of this, so that we can phrase things in terms of posets; up to homeomorphism there is no difference.

\begin{definition}[Vietoris--Rips complex]
Let $X$ be a metric space. Let $P=\mathcal{P}_{\mathrm{fin}}(X)$ be the poset of all non-empty finite subsets of $X$, and for each $t\in\R\cup\{\infty\}$ let $P_t$ be the subposet of those $S\in P$ with $\diam(S)\le t$. The \emph{Vietoris--Rips complex} $\Rips_t(X)$ with parameter $t$ is the geometric realization
\[
\Rips_t(X) \coloneqq |P_t|\text{.}
\]
\end{definition}

Here the \emph{geometric realization} of a poset $(P,\le)$ is the simplicial complex $|P|$ with vertex set $P$ and a $k$-simplex for each chain $p_0<\cdots<p_k$, with face relation given by taking subchains.

Note that $\Rips_\infty(X)$ is always contractible since the poset $P$ is directed, indeed any union of finite subsets is finite. Also, if $X$ is bounded then $\Rips_t(X)$ is contractible for $t$ at least the diameter of $X$, since it is a simplex. The question of which unbounded $X$ admit contractible $\Rips_t(X)$ for some $t<\infty$ is surprisingly difficult in general. For example, for the seemingly easy example of the integer lattice $\Z^n$ in $\R^n$ with the usual $L^1$ metric, $\Rips_t(\Z^n)$ is contractible for large enough $t$, but this was open for a long time until proved by Virk relatively recently in \cite{virk25}; see also \cite{zaremsky26,gupta} for improvements on ``for large enough $t$''.

The $L^1$ metric on $\Z^n$ is an example of a word metric on a group, and more generally investigating contractibility of Vietoris--Rips complexes of groups is a well established endeavor. For example, Rips's proof that hyperbolic groups have contractible Vietoris--Rips complexes was prominent enough to warrant the terminology ``Vietoris--Rips complex'' (Vietoris first studied these complexes \cite{vietoris27}, and Rips's proof is explained for example in \cite[Proposition~III.$\Gamma$.3.23]{bridson99}). See \cite{varisco21,zaremsky22} for more on this, including an approach using Bestvina--Brady discrete Morse theory that takes $\diam$ as an initial measurement.

One particularly useful feature of Vietoris--Rips complexes of groups is that a finitely generated group $G$ always admits a nice (i.e., proper cocompact) action on each $\Rips_t(G)$ by left translation. Indeed, since the word metric $d$ for $G$ is defined by $d(g,h)$ equaling the word length of $g^{-1}h$, this is invariant under left translation by $G$, and finiteness of the generating set makes this action proper and cocompact. In particular, thanks to Brown's finiteness criterion from \cite{brown87}, Vietoris--Rips complexes always reveal finiteness properties of $G$, in the sense that $G$ is of type $\F_n$ if and only if $(\Rips_t(G))_t$ is essentially $(n-1)$-connected, and type $\FP_n$ if and only if $(\Rips_t(G))_t$ is essentially $(n-1)$-acyclic. We will not dwell on this application since our main groups of interest all have the strongest possible finiteness properties anyway, but see \cite[Lemma~6.3]{zaremsky22} for more details on this.

A somewhat obvious, but useful, sufficient condition for $\Rips_t(X)$ to be highly connected is the following.

\begin{proposition}\label{prop:filter}
Let $X$ be a metric space, and let $\mathcal{Y}$ be a family of subsets of $X$ such that every finite subset of $X$ lies in some $Y\in\mathcal{Y}$. View each $Y$ as a metric space with the induced metric. If $\Rips_t(Y)$ is $(n-1)$-connected for all $Y\in\mathcal{Y}$ then $\Rips_t(X)$ is $(n-1)$-connected.
\end{proposition}

\begin{proof}
We want to prove that every map from a sphere of dimension at most $n-1$ to $\Rips_t(X)$ is nullhomotopic. Since spheres are compact, any such map factors through the inclusion of some finite subcomplex $F$ into $\Rips_t(X)$. The vertices of $F$ form a finite subset of $X$, so our hypotheses ensure that this inclusion factors through $\Rips_t(Y)\to \Rips_t(X)$ for some $Y\in\mathcal{Y}$, hence the map is nullhomotopic.
\end{proof}

Finally, for ease of reference let us record the following standard fact about geometric realizations of posets that we will use frequently in what follows.

\begin{cit}\cite[Subsection~1.5]{quillen78}\label{cit:quillen_poset}
Let $(P,\le)$ be a poset and $\phi\colon P\to P$ a poset map. If there exists $p_0\in P$ such that $p\le \phi(p)\ge p_0$ for all $p\in P$, then $|P|$ is contractible. If there exists $p_0\in P$ such that $p\ge \phi(p)\le p_0$ for all $p\in P$, then $|P|$ is contractible.
\end{cit}

\section{Discrete Morse theory}\label{sec:morse}

Let us set up the type of discrete Morse theory we will need, which is a special case of Bestvina--Brady discrete Morse theory as in \cite{bestvina97}. Let $Y$ be a simplicial complex and $f\colon Y^{(0)}\to\R$ a function on the vertex set. If the image of $f$ is closed and discrete, and for each edge of $Y$ the $f$-values of its endpoints are distinct, we call $f$ a \emph{Morse function}. For $t\in\R$ let $Y^{f\le t}$ be the \emph{sublevel complex} of $Y$ consisting of all simplices whose vertices all have $f$-value at most $t$. The \emph{descending link} $\dlk(v)$ of a vertex $v$ is the subcomplex of $Y$ consisting of all simplices $\sigma$ such that $\sigma\cup\{v\}$ is a simplex, $v\not\in\sigma$, and $f(w)<f(v)$ for all vertices $w$ of $\sigma$.

The following Morse lemma essentially follows from \cite[Corollary~2.6]{bestvina97}. For a more direct reference, it is immediate from \cite[Corollary~1.11]{zaremsky22} together with the Whitehead theorem.

\begin{lemma}[Morse lemma]\label{lem:morse}
Let $-\infty\le t<s\le\infty$. If $\dlk(v)$ is $(n-1)$-connected for all vertices $v$ of $Y$ with $t<f(v)\le s$, then the inclusion $Y^{f\le t}\to Y^{f\le s}$ induces an isomorphism in $\pi_k$ for all $k\le n-1$ and a surjection in $\pi_n$. If $\dlk(v)$ is contractible for all $v$ with $t<f(v)\le s$, then $Y^{f\le t}\to Y^{f\le s}$ is a homotopy equivalence.
\end{lemma}

Intuitively, the Morse function provides an ``order'' in which to glue new vertices in, to build up from $Y^{f\le t}$ to $Y^{f\le s}$, in a way that amounts to always coning off descending links up to homotopy. Thus high connectivity of descending links ensures that lower homotopy groups do not change, and contractible descending links ensure the homotopy type does not change.

\medskip

Now let us discuss our specific Morse function of interest. Let $t\in\R$. Let $X$ be a graph whose edges all have length $1$, viewed as a metric space with the path metric. Assume there is a uniform bound on the cardinality of balls of radius $t$. The quintessential example for us is the Cayley graph of a finitely generated group. We will abuse notation and also write $X$ for the set of vertices of $X$ with the induced metric. A vertex of $\Rips_t(X)$ is a non-empty finite subset $S$ of $X$ whose elements are pairwise within distance $t$ of each other (recall that when we write $\Rips_t(X)$ we mean the barycentric subdivision of what is usually called $\Rips_t(X)$). Fix a uniform strict upper bound $N$ on the cardinality of such $S$ (for example one more than the bound on the cardinality of balls of radius $t$ works). Now fix $x_0\in X$ (for example in a Cayley graph think $x_0=1$), write $\len(g)=d(g,x_0)$ for $g\in X$ and $\len_{min}(S)$ for the minimum $\len$ value on a subset $S$, and define
\[
f(S) \coloneqq \len_{min}(S) + \frac{|S|}{N}\text{.}
\]

\begin{lemma}
The function $f$ is a Morse function.
\end{lemma}

\begin{proof}
Since distances in $X$ are integer valued, and $|S|<N$, we see that the image of $f$ is closed and discrete in $\R$. An edge in $\Rips_t(G)$ has endpoints $S_1$ and $S_2$ for some $S_1 \subsetneq S_2$, and we must show that $f(S_1)\ne f(S_2)$. Indeed, either some element of $S_2$ is strictly closer to the identity than all elements of $S_1$, hence $f(S_2)<f(S_1)$, or else $S_1$ and $S_2$ have the same minimum distance to the identity and hence $f(S_2)>f(S_1)$.
\end{proof}

In fact this proof also reveals the descending links. The descending link of $S$ is the subcomplex of $\Rips_t(X)$ spanned by all vertices $S'$ such that, as subsets of $X$, either $S'$ contains $S$ and has strictly smaller minimum distance to $x_0$, or else $S'$ is strictly contained in $S$ and has the same minimum distance to $x_0$. The vertices of the first form span the \emph{descending coface link} $\dclk(S)$, denoted $D_t(S)$ in the introduction, and those of the second form span the \emph{descending face link} $\dflk(S)$ (the terminology comes from recalling how we are working with the barycentric subdivision of what is usually called the Vietoris--Rips complex). Note that $\dlk(S) = \dflk(S) * \dclk(S)$, so if $\dflk(S)$ is $(n-1)$-connected and $\dclk(S)$ is $(m-1)$-connected then $\dlk(S)$ is $(n+m)$-connected, and if either factor is contractible then the whole descending link is contractible.

\begin{lemma}\label{lem:not_flat}
If $\len$ is not constant on $S$, then $\dflk(S)$, hence $\dlk(S)$, is contractible.
\end{lemma}

\begin{proof}
Let $S_0\subsetneq S$ be the subset of those elements minimizing $\len$, so $S_0\in \dflk(S)$. For any $S'\subsetneq S$ in $\dflk(S)$ we have that $S'\cap S_0\ne\emptyset$, so the poset map $\phi \colon S'\mapsto S'\cap S_0$ on the poset whose geometric realization is $\dflk(S)$ satisfies $S'\supseteq \phi(S')\subseteq S_0$. By Citation~\ref{cit:quillen_poset} we conclude that $\dflk(S)$, and hence $\dlk(S)$, is contractible.
\end{proof}

Now if one wishes to use $f$ to prove that $\Rips_t(X)$ is highly connected, the remaining work involves proving high connectivity of descending links of ``flat'' $S$, meaning those $S$ across which $\len$ is constant (this includes the $|S|=1$ case). Since in this case $\dflk(S)$ consists of all proper non-empty subsets of $S$, and hence is a $(|S|-2)$-sphere, this is really about understanding high connectivity of $\dclk(S)$. The following has Theorem~\ref{thrm:main_general} as a special case.

\begin{corollary}\label{cor:general_morse}
If $\dclk(S)$ is $(n-|S|-1)$-connected for all non-empty finite $\{x_0\}\ne S\subseteq X$ such that $\diam(S)\le t$ and $\len$ is constant on $S$, then $\Rips_t(X)$ is $(n-1)$-connected. If $\dclk(S)$ is contractible for all such $S$, then $\Rips_t(X)$ is contractible.
\end{corollary}

\begin{proof}
Every finite subset of $X$ lies in a sublevel complex $\Rips_t(X)^{f\le r}$ for some $r\ge \frac{1}{N}$. Thus by Proposition~\ref{prop:filter} it suffices to prove that the $\Rips_t(X)^{f\le r}$ are $(n-1)$-connected for all $r\ge \frac{1}{N}$. When $r=\frac{1}{N}$ this is just $\{x_0\}$, which is contractible, so by the Morse lemma it suffices to prove that $\dlk(S)$ is $(n-2)$-connected for all $S$ with $f(S)>\frac{1}{N}$, i.e., for all $S\ne\{x_0\}$. If $\len$ is non-constant on $S$ then Lemma~\ref{lem:not_flat} says $\dlk(S)$ is contractible. If $\len$ is constant on $S$ then our hypothesis says $\dclk(S)$ is $(n-|S|-1)$-connected, and $\dflk(S)$ is $(|S|-3)$-connected being a $(|S|-2)$-sphere, so $\dlk(S)$ is $(n-2)$-connected as desired. The contractibility statement is immediate from Whitehead's theorem.
\end{proof}

We can also now prove our sufficient condition for finite presentability, Theorem~\ref{thrm:fp}.

\begin{proof}[Proof of Theorem~\ref{thrm:fp}]
We want to show that $\Rips_t(G)$ is simply connected, which since $G$ acts geometrically on $\Rips_t(G)$ implies $G$ is finitely presented \cite{brown84}. We know that $\dclk(\{g\})$ is connected for all $g\ne 1$ and that $\dclk(\{g,g'\})$ is non-empty for all $g,g'\in G$ with $d(g,g')=t$ and $\len(g)=\len(g')$. By Corollary~\ref{cor:general_morse} all that is left to check is that $\dclk(\{g,g'\})$ is non-empty for all $g\ne g'$ in $G$ with $d(g,g')<t$ and $\len(g)=\len(g')$. Indeed, choose some element $h\in G$ with $d(h,g)=1$ and $\len(h)=\len(g)-1$. Now $d(h,g')\le t$, and so $\{g,g',h\}\in \dclk(\{g,g'\})$.

Now suppose the Cayley graph of $G$ has no odd-length cycles and there exists some odd $t\in\N$ such that $\dclk(\{g\})$ is connected for all $g\ne 1$. The missing assumption is that $\dclk(\{g,g'\})$ is non-empty for all $g,g'\in G$ with $d(g,g')=t$ and $\len(g)=\len(g')$. But this case never happens since a geodesic triangle with vertices $1,g,g'$ would be a cycle of length $t+2\len(g)$, which is odd.
\end{proof}

\section{Two-dimensional RAAGs}\label{sec:raags}

Now we focus on our test case, of two-dimensionsional RAAGs. In this section we recall some background on RAAGs, and establish some properties of two-dimensional RAAGs. Let $\Gamma$ be a finite simplicial graph. The \emph{right-angled Artin group (RAAG)} on $\Gamma$ is
\[
A_\Gamma \coloneqq \langle V(\Gamma) \mid vw=wv \text{ for all } \{v,w\}\in E(\Gamma)\rangle\text{.}
\]
For example if $\Gamma$ is the complete graph on $n$ vertices then $A_\Gamma=\Z^n$, and if $\Gamma$ is the edgeless graph on $n$ vertices then $A_\Gamma=F_n$.

Write $d$ for the word metric on $A_\Gamma$ coming from the generating set $V(\Gamma)$, and write $\len$ for the standard word length, i.e., $\len(g)=d(g,1)$. Note that since the defining relators $vwv^{-1}w^{-1}$ have even length, all relators have even length, i.e., the Cayley graph $\Cay(A_\Gamma)\coloneqq \Cay(A_\Gamma,V(\Gamma))$ of $A_\Gamma$ has no odd-length cycles. In particular if $g$ and $h$ are adjacent vertices then $\len(g)=\len(h)\pm 1$ (as opposed to $\len(g)=\len(h)$).

\begin{definition}
A RAAG $A_\Gamma$ is \emph{two-dimensional} if $\Gamma$ has no $3$-cliques, where a \emph{$3$-clique} is a triple of pairwise adjacent vertices. We will also call $\Gamma$ \emph{triangle-free} in this case.
\end{definition}

This terminology comes from the fact that this is the condition for $A_\Gamma$ to have cohomological dimension at most $2$. Interesting examples include when $\Gamma$ is an $n$-cycle graph ($n\ge 4$), a tree, or a bipartite graph.

It is a standard fact that the Cayley graph $\Cay(A_\Gamma)$ is the $1$-skeleton of a CAT(0) cube complex and hence $\Cay(A_\Gamma)$ a \emph{median graph}, meaning one in which for any three vertices $x_1,x_2,x_3$ there is a unique vertex $m=\med(x_1,x_2,x_3)$, their \emph{median}, such that $m$ lies on a geodesic edge path from $x_i$ to $x_j$ for each $i\ne j$. See \cite{genevois} for much more on the connection between median graphs and CAT(0) cube complexes.

\begin{observation}[Squares are visible]\label{obs:no_weird_squares}
Let $g\in A_\Gamma$ and $a,b,c,d\in V(\Gamma)^\pm$ such that $abcd=1$ and the vertices $ga,gab,gabc,gabcd=g$ form an induced square in $\Cay(A_\Gamma)$, i.e., a subgraph isometric to a $4$-cycle graph. Then $c=a^{-1}$ and $d=b^{-1}$, so $a$ and $b$ commute, and moreover this square has a unique vertex at which $\len$ is maximized and a unique vertex at which $\len$ is minimized.
\end{observation}

\begin{proof}
There is a well defined homomorphism $A_\Gamma\to\Z$ sending $a$ to $1$ and sending all generators not equal to $a^\pm$ to $0$. Since $abcd=1$, this shows that at least one of $b$, $c$, or $d$ must equal $a^{-1}$, and since our vertices form an induced square we must have $c=a^{-1}$. An analogous argument shows $d=b^{-1}$, and now $aba^{-1}b^{-1}=1$ shows $a$ and $b$ commute. Next, standard facts about the word problem in RAAGs (see, e.g., \cite[Subsection~2.3]{charney}) show that $\len(g)-\len(gb)=\len(ga)-\len(gab)$. Hence if $\len(g)<\len(ga)$ and $\len(g)<\len(gb)$ then the unique vertex maximizing $\len$ is $gab$ and the unique vertex minimizing $\len$ is $g$, if $\len(g)<\len(ga)$ and $\len(g)>\len(gb)$ then the unique such vertices are $ga$ and $gb$, if $\len(g)>\len(ga)$ and $\len(g)<\len(gb)$ then they are $gb$ and $ga$, and finally if $\len(g)>\len(ga)$ and $\len(g)>\len(gb)$ then they are $g$ and $gab$.
\end{proof}

\begin{definition}[Descending neighbor, label]
For $g \in A_\Gamma$, a \emph{descending neighbor} of $g$ is a vertex $h \in A_\Gamma$ such that $d(g,h)=1$ and $\len(h) = \len(g) - 1$. The \emph{label} $v\in V(\Gamma)^\pm$ of a descending neighbor $h$ of $g$ is $v\coloneqq h^{-1}g$, that is, the generator or inverse generator that one multiplies on the right to move from $h$ to $g$. Note that distinct descending neighbors of $g$ must have different labels.
\end{definition}

\begin{proposition}\label{prop:median_nbrs}
Let $g\in A_\Gamma$ with $\len(g) \geq  2$, and let $h_1, h_2$ be distinct descending neighbors of $g$. Then the median $m \coloneqq \med(h_1, h_2, 1)$ satisfies the following:
\begin{itemize}
    \item $d(h_1, m) = d(h_2, m) = 1$,
    \item $\len(m) = \len(g) - 2$,
    \item $d(g, m) = 2$, and
    \item the labels of $h_1$ and $h_2$ (relative $g$) commute.
\end{itemize}
\end{proposition}

\begin{proof}
First note that $d(h_1, h_2) = 2$. Indeed, $d(h_1, h_2) \leq 2$ by the triangle inequality, $d(h_1, h_2) \neq 0$ because $h_1 \neq h_2$, and $d(h_1, h_2) \neq 1$ because otherwise $g,h_1,h_2$ would yield a $3$-cycle in $\Cay(A_\Gamma)$ (which has no odd-length cycles). Now by properties of the median we have:
\[
d(h_1, m) + d(m, h_2) = 2 \qquad d(h_1, m) + d(m, 1) = \len(g) - 1 = d(h_2, m) + d(m, 1) \text{.}
\]
If $d(h_2, m) = 0$ then $h_2 = m$ and we get $\len(h_1) = d(h_1, h_2) + \len(h_2) = 2 + (\len(g) - 1) = \len(g) + 1$, which is not the case, so $h_2 \neq m$, and an identical argument shows $h_1 \neq m$. Since $d(h_1, h_2) = 2$ we have that $d(h_1, m) = 1 = d(h_2, m)$. This further implies that $\len(m) = \len(g) - 2$. Finally, we have $d(g, m) \leq 2$ by the triangle inequality, and $\len(g) \leq d(g, m) + \len(m)$ implies that $d(g, m) \geq \len(g) - \len(m) = \len(g) - (\len(g)-2) = 2$.

Finally, these results show that the vertices $g,h_1,h_2,m$ form an induced square in $\Cay(A_\Gamma)$, and hence the labels of $h_1$ and $h_2$ commute by Observation~\ref{obs:no_weird_squares}.
\end{proof}

Since $\Cay(A_\Gamma)$ is connected, every non-trivial element of $A_\Gamma$ has at least one descending neighbor. The following shows that if $\Gamma$ is triangle-free then every element has at most two descending neighbors:

\begin{corollary}\label{cor:two_nbrs}
If some $g\in A_\Gamma$ has more than two descending neighbors then $\Gamma$ has a $3$-clique.
\end{corollary}

\begin{proof}
Say $h_1,h_2,h_3$ are distinct descending neighbors of $g$, with labels $v_1,v_2,v_3$ respectively. By Proposition~\ref{prop:median_nbrs} the $v_i$ pairwise commute, which yields a $3$-clique in $\Gamma$.
\end{proof}

We will also be concerned with when multiple elements have a descending neighbor in common. Of course a necessary condition for this is that the elements have the same distance to the identity and are distance $2$ from each other. In the triangle-free case it turns out this is also sufficient, and the common descending neighbor is unique. We will prove this in steps. For the first step we do not need to assume $\Gamma$ is triangle-free.

\begin{lemma}\label{lem:exactly_two}
Let $g_1,g_2\in A_\Gamma$ with $\len(g_1)=\len(g_2)$ and $d(g_1,g_2)=2$. Then $g_1$ and $g_2$ have a unique descending neighbor in common.
\end{lemma}

\begin{proof}
First we show existence. Since $d(g_1,g_2)=2$ there is a length-$2$ edge path from $g_1$ to $g_2$, say with middle vertex $h$. If $\len(h)=\len(g_1)-1$ then $h$ is a descending neighbor of both $g_1$ and $g_2$ and we are done. We cannot have $\len(h)=\len(g_1)$, since $\Cay(A_\Gamma)$ has no odd-length cycles, so we must have $\len(h)=\len(g_1)+1$. Now $g_1$ and $g_2$ are both descending neighbors of $h$, so by Proposition~\ref{prop:median_nbrs} $\med(g_1,g_2,1)$ is a descending neighbor of both $g_1$ and $g_2$.

For uniqueness, suppose $h_1\ne h_2$ are descending neighbors of $g_1\ne g_2$. Then the vertices $g_1,h_1,g_2,h_2$ form an induced square in $\Cay(A_\Gamma)$ with two vertices maximizing $\len$, which violates Observation~\ref{obs:no_weird_squares}.
\end{proof}

Now we focus on the triangle-free case. It is easier to deal with large collections of elements:

\begin{lemma}\label{lem:at_least_four}
Suppose $\Gamma$ has no $3$-cliques. Let $g_1,\dots,g_k\in A_\Gamma$ for $k\ge 4$ such that $d(g_i,g_j)=2$ for all $1\le i<j\le k$ and $\len(g_1)=\cdots=\len(g_k)$. Then there exists a unique $h\in A_\Gamma$ that is a descending neighbor of every $g_i$.
\end{lemma}

\begin{proof}
First we prove existence. Let $h_1,\dots,h_\ell$ be all the elements that arise as descending neighbors of some $g_i$. Consider the bipartite graph $\Delta$ with blue vertices $g_1,\dots,g_k$, red vertices $h_1,\dots,h_\ell$, and an edge from $g_j$ to $h_i$ whenever $h_i$ is a descending neighbor of $g_j$. By Lemma~\ref{lem:exactly_two} each pair of blue vertices is adjacent to a unique common red vertex (so in particular $\Delta$ is connected), and by Corollary~\ref{cor:two_nbrs} each blue vertex has degree at most $2$. If some blue vertex has degree $1$ then all other blue vertices must be adjacent to the one red vertex it is adjacent to, and we are done. Now suppose all blue vertices have degree $2$. Let $\Delta'$ be the simplicial graph with vertex set $h_1,\dots,h_\ell$ and an edge from $h_i$ to $h_j$ whenever $h_i$ and $h_j$ are both adjacent in $\Delta$ to some blue vertex (so $\Delta'$ essentially comes from erasing the blue vertices from $\Delta$). Now the hypotheses ensure that each pair of edges in $\Delta'$ share an endpoint. Since $k\ge 4$ we know $\Delta'$ is not a $3$-cycle graph, so the only way this can happen is if it is a star graph, which means some red vertex in $\Delta$ is adjacent to every blue vertex, and so we are done.

Uniqueness is immediate by applying Lemma~\ref{lem:exactly_two} to, say, $g_1$ and $g_2$.
\end{proof}

Finally we deal with the case of exactly three elements, after a preliminary technical lemma.

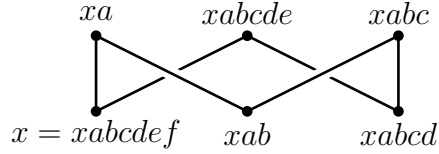
\begin{figure}[htb]
\begin{tikzpicture}[line width=1pt]
\draw (0,0) -- (2,1) -- (4,0);
\draw[white,line width=5pt] (0,1) -- (2,0) -- (4,1);
\draw (0,0) -- (0,1) -- (2,0) -- (4,1) -- (4,0);
\filldraw (0,0) circle (1.5pt);
\filldraw (0,1) circle (1.5pt);
\filldraw (2,0) circle (1.5pt);
\filldraw (2,1) circle (1.5pt);
\filldraw (4,1) circle (1.5pt);
\filldraw (4,0) circle (1.5pt);
\node at (0,-0.3) {$x=xabcdef$};
\node at (0,1.3) {$xa$};
\node at (2,-0.3) {$xab$};
\node at (4,1.3) {$xabc$};
\node at (4,-0.3) {$xabcd$};
\node at (2,1.3) {$xabcde$};
\end{tikzpicture}
\caption{The situation in Lemma~\ref{lem:hex}.  For context this is a subgraph of $\Cay(A_\Gamma)$, and the measurement $\len$ is represented by vertical height.}
\label{fig:hex}
\end{figure}

\begin{lemma}\label{lem:hex}
Let $x\in A_\Gamma$ and $a,b,c,d,e,f\in V(\Gamma)^\pm$ such that $abcdef=1$ and the six vertices
\[
xa,xab,xabc,xabcd,xabcde,xabcdef=x
\]
are pairwise distinct. Suppose that $\len(x)=\len(xab)=\len(xabcd)$ and $\len(xa)=\len(xabc)=\len(xabcde)=\len(x)+1$. Then $\Gamma$ has a $3$-clique.
\end{lemma}

\begin{proof}
See Figure~\ref{fig:hex} for an idea of this setup. Applying Proposition~\ref{prop:median_nbrs} to $xa$, $xabc$, and $xabcde$ we see that $ab=ba$, $cd=dc$, and $ef=fe$. Looking at the map $A_\Gamma\to\Z$ sending $a$ to $1$ and all generators not equal to $a^\pm$ to $0$, in order for $abcdef=1$ to hold we must have that $a^{-1}\in\{b,c,d,e,f\}$. Since our six vertices are distinct, $a^{-1}\ne b$ and $a^{-1}\ne f$. Since $\len(xbac)=\len(xabc)=\len(x)+1$ and Proposition~\ref{prop:median_nbrs} says $\len(xb)=\len(x)-1$, we see that $a^{-1}\ne c$. Similarly $a^{-1}\ne e$ (by symmetry). We conclude that $a^{-1}=d$. By an analogous argument $b^{-1}=e$ and $c^{-1}=f$. Now $a$, $b$, and $c$ all commute with each other and we have a $3$-clique in $\Gamma$.
\end{proof}

\begin{lemma}\label{lem:exactly_three}
Suppose $\Gamma$ has no $3$-cliques. Let $g_1,g_2,g_3\in A_\Gamma$ such that $d(g_i,g_j)=2$ for all $1\le i<j\le 3$ and $\len(g_1)=\len(g_2)=\len(g_3)$. Then there exists a unique $h\in A_\Gamma$ that is a descending neighbor of every $g_i$.
\end{lemma}

\begin{proof}
By Proposition~\ref{prop:median_nbrs} we can choose $h_1,h_2,h_3$ such that $h_i$ and $h_{i+1}$ are both descending neighbors of $g_i$ (subscripts mod $3$) for all $i$. Applying Lemma~\ref{lem:hex} to the labels, we see that $h_1,h_2,h_3$ cannot be pairwise distinct, say without loss of generality that $h_3=h_2$. Now by Observation~\ref{obs:no_weird_squares}, since $\len(h_1)=\len(h_2)$ and $\len(g_1)=\len(g_3)$ we see that $h_1$, $g_1$, $h_2$, and $g_3$ cannot form an induced square, so $h_1=h_2$ and we have found a common descending neighbor. Uniqueness is immediate by applying Lemma~\ref{lem:exactly_two} to, say, $g_1$ and $g_2$.
\end{proof}

\begin{corollary}\label{cor:unique}
Suppose $\Gamma$ has no $3$-cliques. Let $S\subseteq A_\Gamma$ be finite such that $|S|\ge 2$, $d(g,g')=2$ for all $g\ne g'$ in $S$, and $\len$ is constant on $S$. Then the elements of $S$ have a unique common descending neighbor.
\end{corollary}

\begin{proof}
This is the combination of Lemmas~\ref{lem:exactly_two}, \ref{lem:at_least_four}, and~\ref{lem:exactly_three}, for $|S|=2$, $|S|\ge 4$, and $|S|=3$ respectively.
\end{proof}

\section{Vietoris--Rips complexes of two-dimensional RAAGs}\label{sec:VR_raags}

In this section we prove Theorem~\ref{thrm:main_raags}, that $\Rips_2(A_\Gamma)$ is contractible for $\Gamma$ triangle-free. For the rest of the paper, fix a finite simple graph $\Gamma$ with no $3$-cliques, and a vertex $S\ne\{1\}$ of $\Rips_2(A_\Gamma)$ such that $\len$ is constant on $S$, so with Corollary~\ref{cor:general_morse} in mind we want to prove that $\dclk(S)$ is contractible. The $|S|=1$ and $|S|\ge 2$ cases turn out to work differently, since thanks to Corollary~\ref{cor:unique} we know that in the latter case the elements of $S$ have a unique common descending neighbor. Let us handle this situation first.

\begin{proposition}\label{prop:unique_d_nbr}
If $\len$ is constant on $S$ and the elements of $S$ have a unique common descending neighbor, then $\dclk(S)$ is contractible. In particular this happens if $\len$ is constant on $S$ and $|S|\ge 2$.
\end{proposition}

\begin{proof}
Write $n\coloneqq \len(g)$ for some/any $g\in S$. Say the unique common descending neighbor is $h$, and we claim that for all $S'\in \dclk(S)$ also $S'\cup\{h\}\in \dclk(S)$. Since $\len(h)<n$, $S'\cup\{h\}$ will be descending as soon as it is a legitimate element, i.e., we just need to check that $h$ is within distance $2$ of every element of $S'$. Note that for $S'\supseteq S$ to be in $\dclk(S)$ there must be some $g_0\in S'$ with $\len(g_0)<n$, and we must have $d(g',g'')\le 2$ for all $g',g''\in S'$. In particular $\len(g_0)$ can only be $n-1$ or $n-2$.

First suppose $\len(g_0)=n-1$ for some $g_0\in S'$. Since $\Cay(A_\Gamma)$ has no odd-length cycles, the distance from $g_0$ to any element of $S$ must be odd, and so the only possibility is distance $1$. Thus $g_0$ is a descending neighbor of every element of $S$, so by our uniqueness assumption $g_0=h$. Thus $S'\cup\{h\}=S'$ and we are done in this case.

Now assume that for each $g'\in S'$ the value $\len(g')$ is either $n$ or $n-2$, call these the \emph{high} and \emph{low} elements respectively. We know that all the elements of $S$ are high, and that $S'$ has at least one low element. We must prove that $d(g',h)=1$ for all $g'\in S'$. First we prove this for the low elements; fix a low element $g_0\in S'$. For each high $g'\in S'$ (so in particular for all $g'\in S$), let $h_{g'}$ be the point in the middle of some length-$2$ path from $g'$ to $g_0$. The $h_{g'}$ all have $\len(h_{g'})=n-1$, and $g_0$ is a common descending neighbor of all of them. Let $T=\{h_{g'}\mid g'\in S'$ with $\len(g')=n\}\cup\{h\}$, so the elements of $T$ are pairwise distance $2$ from each other, and $\len$ is constant $\len=n-1$ on $T$. By Corollary~\ref{cor:unique} either $|T|=1$ or the elements of $T$ have a unique common descending neighbor. If $|T|=1$ then $d(h,g_0)=1$ and we are done. Suppose instead that the elements of $T$ have a unique common descending neighbor $g_0'$. Since $g)$ and $g_0'$ are both descending neighbors of all the $h_{g'}$, by Corollary~\ref{cor:unique} either $g_0'=g_0$ and again we are done, or else the $h_{g'}$ are all equal to each other. In this case they provide a common descending neighbor of $S$, and so by our hypothesis they equal $h$ and once again we are done. We note that, now that we know all the low elements have distance $1$ to $h$, they are descending neighbors of $h$, so by Corollary~\ref{cor:two_nbrs} there are at most two low elements.

Now we must prove that every high element of $S'$ has distance $1$ to $h$. Fix a high element $g_1$; since $h$ is a descending neighbor of every element of $S$ we can assume $g_1\not\in S$. For each $g\in S$, since $d(g_1,g)=2$ and $\len(g_1)=\len(g)=n$, by Proposition~\ref{prop:median_nbrs} there is a length-$2$ path from $g_1$ to $g$ passing through a point $h_g$ with $\len$ value $n-1$. By Observation~\ref{obs:no_weird_squares} this path, and hence $h_g$, is uniquely determined by $g_1$ and $g$. If $|S|=1$ then by uniqueness of $h$, $h_g=h$ and we are done, so assume $|S|\ge 2$. If the $h_g$ all equal each other then they are a common descending neighbor of every element of $S$, so by uniqueness of $h$ they equal $h$ and we are done. Similarly if some $h_g$ equals $h$ then we are done, so at this point we can assume that for some $g\ne g'$ in $S$ the elements $h_g$, $h_{g'}$, and $h$ are all distinct. Now the elements $g_1$, $h_g$, $g$, $h$, $g'$, and $h_{g'}$ satisfy the hypotheses of Lemma~\ref{lem:hex}, contradicting that $\Gamma$ has no $3$-cliques. Hence the final case cannot happen, and we conclude that in fact $d(g_1,h)=1$ as desired.

Now that we know that for all $S'\in \dclk(S)$ also $S'\cup\{h\}\in \dclk(S)$, we have a well defined poset map $\phi\colon S'\mapsto S'\cup\{h\}$ on the underlying poset whose geometric realization is $\dclk(S)$. Since $S'\subseteq \phi(S')\supseteq S\cup\{h\}$, and $S\cup\{h\}\in \dclk(S)$, Citation~\ref{cit:quillen_poset} says $\dclk(S)$ is contractible.

The ``in particular'' part of the statement follows from Corollary~\ref{cor:unique}.
\end{proof}

Next we handle the $|S|=1$ case, which is more difficult thanks to the potential lack of uniqueness of descending neighbors. First we need one last technical lemma (we reiterate that $\Gamma$ has no $3$-cliques).

\begin{figure}[htb]
\begin{tikzpicture}[line width=1pt]
\draw (0,0) -- (-1,1) -- (-1,2) -- (0,3) -- (1,2) -- (1,1) -- (0,0);
\filldraw (0,0) circle (1.5pt);
\filldraw (-1,1) circle (1.5pt);
\filldraw (-1,2) circle (1.5pt);
\filldraw (0,3) circle (1.5pt);
\filldraw (1,2) circle (1.5pt);
\filldraw (1,1) circle (1.5pt);
\node at (0,-0.3) {$x=xabcdef$};
\node at (-1.4,1) {$xa$};
\node at (-1.5,2) {$xab$};
\node at (0,3.3) {$xabc$};
\node at (1.6,2) {$xabcd$};
\node at (1.7,1) {$xabcde$};
\end{tikzpicture}
\caption{The situation in Lemma~\ref{lem:straight_hex}. For context this is a subgraph of $\Cay(A_\Gamma)$, and the measurement $\len$ is represented by vertical height.}
\label{fig:straight_hex}
\end{figure}
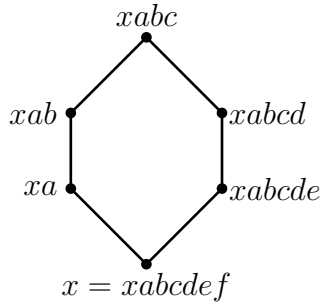

\begin{lemma}\label{lem:straight_hex}
Let $x\in A_\Gamma$ and $a,b,c,d,e,f\in V(\Gamma)^\pm$ such that $abcdef=1$ and the six vertices
\[
xa,xab,xabc,xabcd,xabcde,xabcdef=x
\]
are pairwise distinct. Suppose $\len(x)<\len(xa)<\len(xab)<\len(xabc)>\len(xabcd)>\len(xabcde)>\len(x)$. Then $a,b,c,d,e,f\in\{v,w\}^\pm$ for some commuting $v,w\in V(\Gamma)$. In particular $xa$ and $xabcde$ are either both descending neighbors of $xab$ or of $xabcd$.
\end{lemma}

\begin{proof}
See Figure~\ref{fig:straight_hex} for an idea of this setup. Applying Proposition~\ref{prop:median_nbrs} to the vertex $xabc$ and its descending neighbors $xab$ and $xabcd$ we see that $cd=dc$. In particular $xab$ has descending neighbors $xa$ and $xabd$, so either $d=b^{-1}$ or the same argument says $bd=db$, so in either case $bd=db$. Now $xa$ has descending neighbors $x$ and $xad$ so once again we get that $da=ad$. Having shown that $d$ commutes with $a$, $b$, and $c$, since $abcdef=1$ we see $d$ commutes with $ef$. If $d$ commutes with neither $e$ nor $f$, then since they are generators of a RAAG we know that the subgroup generated by $d,e,f$ splits as a free product $\langle d\rangle*\langle e,f\rangle$, and hence $d$ does not commute with $ef$. Since it does, we conclude that $d$ does commute with at least one of $e$ or $f$, and since it commutes with their product in fact it commutes with both. An analogous argument shows that $c$ commutes with all of $a,b,c,d,e,f$. Since $c\ne d^\pm$ and $a,b,e,f$ all commute with both $c$ and $d$, but $\Gamma$ has no $3$-cliques, we conclude that each of them must equal $c^\pm$ or $d^\pm$, and we are done.
\end{proof}

For example in $\Z^2=\langle a,b\rangle$ the relation $abba^{-1}b^{-1}b^{-1}=1$ satisfies the conditions of Lemma~\ref{lem:straight_hex}.

\begin{proposition}\label{prop:one_point}
If $S=\{g\}$ for $g\ne 1$ then $\dclk(S)$ is contractible.
\end{proposition}

\begin{proof}
By Corollary~\ref{cor:two_nbrs}, $g$ has at most two descending neighbors. If it has exactly one descending neighbor then Proposition~\ref{prop:unique_d_nbr} says $\dclk(\{g\})$ is contractible, so suppose $g$ has exactly two descending neighbors, say $h_1$ and $h_2$. We claim that $\dclk(\{g\})$ is the union of the stars of $\{g,h_1\}$ and $\{g,h_2\}$. Since these stars are contractible and their intersection is the star of $\{g,h_1,h_2\}$, which is also contractible, this will show that $\dclk(\{g\})$ is contractible.

Let $S'\in \dclk(\{g\})$ be arbitrary, so $g\in S'$, $d(g',g'')\le 2$ for all $g',g''\in S'$, and there exists $g_0\in S'$ with $\len(g_0)<\len(g)$. We have to show that either $S'\cup\{h_1\}$ or $S'\cup\{h_2\}$ lies in $\dclk(\{g\})$, i.e., that either all the elements of $S'$ are within distance $2$ of $h_1$, or are all within distance $2$ of $h_2$. If $S'$ contains an element with $\len$ value $\len(g)+1$ then every $g_0\in S'$ satisfying $\len(g_0)<\len(g)$ must actually satisfy $\len(g_0)=\len(g)-1$, and hence $g_0$ is either $h_1$ or $h_2$ and we are done. Now assume $S'$ has no such elements, so the $\len$ values of elements of $S'$ can only be $\len(g)$, $\len(g)-1$, or $\len(g)-2$. The only elements with $\len$ value $\len(g)-1$ that can appear in such an $S'$ are $h_1$ and $h_2$, so in this case we are done. Now assume the elements of $S'$ only have $\len$ values $\len(g)$ (\emph{high} elements) and $\len(g)-2$ (\emph{low} elements).

First suppose $g$ is the only high element of $S'$. Every low element must be a descending neighbor of $h_1$ or $h_2$ (or both), so in particular if there is only one low element then we are done. Now suppose $g_0\ne g_0'$ are low elements such that $g_0$ is a descending neighbor of $h_1$ and $g_0'$ is a descending neighbor of $h_2$. Since $d(g_0,g_0')=2$, by Lemma~\ref{lem:exactly_two} $g_0$ and $g_0'$ have a common descending neighbor $h$. Now the vertices $h,g_0,h_1,g,h_2,g_0'$ satisfy the hypotheses of Lemma~\ref{lem:straight_hex}, and so $g_0$ and $g_0'$ are either both descending neighbors of $h_1$ or of $h_2$. This works for every pair of low elements, so we conclude that either all the low elements are distance $1$ to $h_1$, or all distance $1$ to $h_2$, and we are done with this case.

From now on we assume there is at least one high element in $S'$ other than $g$. First we claim that all the high elements share either $h_1$ or $h_2$ as a descending neighbor. Let $g_1\in S'\setminus\{g\}$ be high. By Proposition~\ref{prop:median_nbrs} there is a path of length $2$ from $g_1$ to $g$ passing through a descending neighbor of $g$, either $h_1$ or $h_2$. If $g_1'$ is a different high element of $S'\setminus\{g\}$, then by Lemma~\ref{lem:hex} $g_1$ and $g_1'$ must both have either $h_1$ or $h_2$ as a descending neighbor. Since a given $g_1$ cannot have both $h_1$ and $h_2$ as descending neighbors by Observation~\ref{obs:no_weird_squares} (since $g$ already does), we conclude that all the high elements share either $h_1$ or $h_2$ as a descending neighbor, and none of them (except $g$) have the other as a descending neighbor. Without loss of generality say they all share $h_1$.

Now it just remains to prove that all the low elements are distance $1$ to $h_1$. Suppose there is a low element $g_0$ that is a descending neighbor of $h_2$ but not of $h_1$. Let $g_1\ne g$ be a high element, so $h_1$ is a descending neighbor of $g_1$ and $h_2$ is not. Since $d(g_1,g_0)=2$ and $g_0$ is not a descending neighbor of $h_1$, there is some $h\not\in\{h_1,h_2\}$ on a length-$2$ path from $g_1$ to $g_0$. Finally, let $m_1=\med(h,h_1,1)$ and $m_2=\med(h_1,h_2,1)$, so thanks to Proposition~\ref{prop:median_nbrs} if the vertices $g_0,h_2,m_2,h_1,m_1,h$ are all distinct then they satisfy the hypotheses of Lemma~\ref{lem:hex}, violating that $\Gamma$ has no $3$-cliques; see Figure~\ref{fig:last_thing} for a picture keeping track of all these vertices. Thus these vertices cannot be all distinct, which considering their $\len$ values and knowing that $h,h_1,h_2$ are distinct and that $g_0$ is not a descending neighbor of $h_1$ tells us that $m_1=m_2$, call it $m$. But now the vertices $m,h,g_0,h_2$ violate Observation~\ref{obs:no_weird_squares}, and we have reached our desired contradiction.
\end{proof}

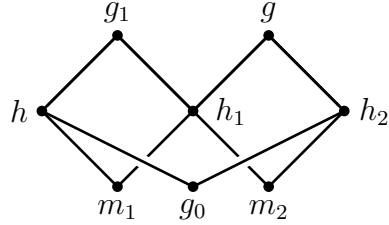
\begin{figure}[htb]
\begin{tikzpicture}[line width=1pt]
\draw (1,-1) -- (2,0) -- (3,-1);
\draw (0,0) -- (1,1) -- (2,0) -- (3,1) -- (4,0);
\draw[white, line width=5pt] (0,0) -- (2,-1) -- (4,0);
\draw (0,0) -- (2,-1) -- (4,0);
\draw (0,0) -- (1,1) -- (2,0) -- (3,1) -- (4,0);
\draw (0,0) -- (1,-1)   (3,-1) -- (4,0);

\filldraw (0,0) circle (1.5pt);
\filldraw (1,1) circle (1.5pt);
\filldraw (2,0) circle (1.5pt);
\filldraw (3,1) circle (1.5pt);
\filldraw (4,0) circle (1.5pt);
\filldraw (2,-1) circle (1.5pt);
\filldraw (1,-1) circle (1.5pt);
\filldraw (3,-1) circle (1.5pt);

\node at (1,1.3) {$g_1$};
\node at (3,1.3) {$g$};
\node at (2.5,0) {$h_1$};
\node at (4.4,0) {$h_2$};
\node at (-0.3,0) {$h$};
\node at (2,-1.3) {$g_0$};
\node at (1,-1.3) {$m_1$};
\node at (3,-1.3) {$m_2$};
\end{tikzpicture}
\caption{A visualization of the last part of the proof of Proposition~\ref{prop:one_point}, before we realize that $m_1=m_2$ (and before we realize that is also impossible).}
\label{fig:last_thing}
\end{figure}

\begin{proof}[Proof of Theorem~\ref{thrm:main_raags}]
By Theorem~\ref{thrm:main_general} it suffices to prove that for all non-empty finite $\{1\}\ne S\subseteq G$ such that $\len$ is constant on $S$, $\dclk(S)$ is contractible. If $|S|\ge 2$ then this follows from Proposition~\ref{prop:unique_d_nbr}, and if $|S|=1$ then this follows from Proposition~\ref{prop:one_point}.
\end{proof}

\bibliographystyle{alpha}

\begin{thebibliography}{CCG{\etalchar{+}}25}

\bibitem[BB97]{bestvina97}
Mladen Bestvina and Noel Brady.
\newblock Morse theory and finiteness properties of groups.
\newblock {\em Invent. Math.}, 129(3):445--470, 1997.

\bibitem[BH99]{bridson99}
Martin~R. Bridson and Andr\'{e} Haefliger.
\newblock {\em Metric spaces of non-positive curvature}, volume 319 of {\em Grundlehren der mathematischen Wissenschaften [Fundamental Principles of Mathematical Sciences]}.
\newblock Springer-Verlag, Berlin, 1999.

\bibitem[Bro84]{brown84}
Kenneth~S. Brown.
\newblock Presentations for groups acting on simply-connected complexes.
\newblock {\em J. Pure Appl. Algebra}, 32(1):1--10, 1984.

\bibitem[Bro87]{brown87}
Kenneth~S. Brown.
\newblock Finiteness properties of groups.
\newblock In {\em Proceedings of the {N}orthwestern conference on cohomology of groups ({E}vanston, {I}ll., 1985)}, volume~44, pages 45--75, 1987.

\bibitem[CCG{\etalchar{+}}25]{chalopin25}
J\'er\'emie Chalopin, Victor Chepoi, Anthony Genevois, Hiroshi Hirai, and Damian Osajda.
\newblock Helly groups.
\newblock {\em Geom. Topol.}, 29(1):1--70, 2025.

\bibitem[Cha]{charney}
Ruth Charney.
\newblock An introduction to right-angled {A}rtin groups.
\newblock arXiv:0610668.

\bibitem[For98]{forman98}
Robin Forman.
\newblock Morse theory for cell complexes.
\newblock {\em Adv. Math.}, 134(1):90--145, 1998.

\bibitem[Gen]{genevois}
Anthony Genevois.
\newblock Why {CAT(0)} cube complexes should be replaced with median graphs.
\newblock arXiv:2309.02070.

\bibitem[GSS]{gupta}
Raju~Kumar Gupta, Sourav Sarkar, and Samir Shukla.
\newblock On the {V}ietoris-{R}ips complexes of integer lattices.
\newblock arXiv:2511.04238.

\bibitem[LS]{li}
Kevin Li and Luis~Jorge {S\'anchez Salda\~na}.
\newblock Contractible {R}ips complexes of groups via metric gluings.
\newblock arXiv:2608.24279.

\bibitem[Qui78]{quillen78}
Daniel Quillen.
\newblock Homotopy properties of the poset of nontrivial {$p$}-subgroups of a group.
\newblock {\em Adv. in Math.}, 28(2):101--128, 1978.

\bibitem[Vie27]{vietoris27}
L.~Vietoris.
\newblock {\"U}ber den h\"oheren {Z}usammenhang kompakter {R}\"aume und eine {K}lasse von zusammenhangstreuen {A}bbildungen.
\newblock {\em Math. Ann.}, 97(1):454--472, 1927.

\bibitem[Vir25]{virk25}
{\v Z}iga Virk.
\newblock Contractibility of the {R}ips complexes of integer lattices via local domination.
\newblock {\em Trans. Amer. Math. Soc.}, 378(3):1755--1770, 2025.

\bibitem[VZ21]{varisco21}
Marco Varisco and Matthew C.~B. Zaremsky.
\newblock Equivariant {M}orse theory on {V}ietoris-{R}ips complexes and universal spaces for proper actions.
\newblock {\em Bull. Lond. Math. Soc.}, 53(6):1724--1739, 2021.

\bibitem[Zar22]{zaremsky22}
Matthew C.~B. Zaremsky.
\newblock Bestvina-{B}rady discrete {M}orse theory and {V}ietoris-{R}ips complexes.
\newblock {\em Amer. J. Math.}, 144(5):1177--1200, 2022.

\bibitem[Zar26]{zaremsky26}
Matthew C.~B. Zaremsky.
\newblock Contractible {V}ietoris-{R}ips complexes of {${\mathbb Z}^n$}.
\newblock {\em Proc. Amer. Math. Soc.}, 154(2):503--508, 2026.

\end{thebibliography}
\newcommand{\etalchar}[1]{$^{#1}$}

\end{document}